\pdfoutput=1
\documentclass[11pt]{article}

\usepackage[utf8]{inputenc}
\usepackage[pdftex]{graphicx,xcolor}
\usepackage{hyperref}
\hypersetup{pdftex,colorlinks=true,allcolors=blue}
\makeatletter
\@ifpackagelater{hyperref}{2012/07/24}{
\hypersetup{unicode,pdftex,colorlinks=true,allcolors=blue,psdextra}
}{
\hypersetup{unicode,pdftex,colorlinks=true,allcolors=blue}
}
\makeatother
\usepackage{array,bookmark,fullpage}
\usepackage[nottoc]{tocbibind}
\usepackage{amsmath,amssymb,amsthm,mathtools,bm,tikz-cd,stmaryrd}
\usepackage[capitalize]{cleveref}

\usepackage{mymacros}

\let\defn\textbf

\def\op{\mathrm{op}}
\mathlig{<-}{\leftarrow}
\DeclareMathOperator\Hom{Hom}

\begin{document}

\title{Amalgamable diagram shapes}
\author{Ruiyuan Chen\thanks{Research partially supported by NSERC PGS D}}
\date{}
\maketitle

\begin{abstract}
A category has the amalgamation property (AP) if every pushout diagram has a cocone, and the joint embedding property (JEP) if every finite coproduct diagram has a cocone.  
We show that for a finitely generated category $\*I$, the following are equivalent: (i) every $\*I$-shaped diagram in a category with the AP and the JEP has a cocone; (ii) every $\*I$-shaped diagram in the category of sets and injections has a cocone; (iii) a certain canonically defined category $\@L(\*I)$ of ``paths'' in $\*I$ has only idempotent endomorphisms.  When $\*I$ is a finite poset, these are further equivalent to: (iv) every upward-closed subset of $\*I$ is simply-connected; (v) $\*I$ can be built inductively via some simple rules.  Our proof also shows that these conditions are decidable for finite $\*I$.
\end{abstract}

\section{Introduction}
\label{sec:intro}

This paper concerns a category-theoretic question which arose in a model-theoretic context.  In model theory, specifically in the construction of Fraïssé limits (see \cite[\S7.1]{H}), one considers a class $\@K$ of structures (in some first-order language) with the following properties:
\begin{itemize}
\item  the \defn{joint embedding property (JEP)}: ($\@K$ is nonempty and) for every two structures $A, B \in \@K$, there are embeddings $f : A -> X$ and $g : B -> X$ into some $X \in \@K$:
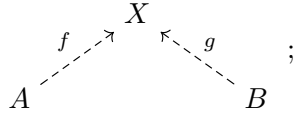
\begin{equation*}
\begin{tikzcd}[row sep=1.5em]
& X \\
A \urar[dashed]{f} && B \ular[dashed][swap]{g}
\end{tikzcd};
\end{equation*}
\item  the \defn{amalgamation property (AP)}: every diagram of embeddings
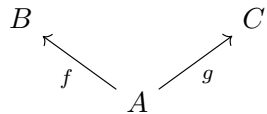
\begin{equation*}
\begin{tikzcd}[row sep=1.5em]
B && C \\
& A \ular{f} \urar[swap]{g}
\end{tikzcd}
\end{equation*}
between structures $A, B, C \in \@K$ can be completed into a commutative diagram
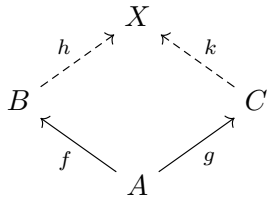
\begin{equation*}
\begin{tikzcd}[row sep=1.5em]
& X \\
B \urar[dashed]{h} && C \ular[dashed][swap]{k} \\
& A \ular{f} \urar[swap]{g}
\end{tikzcd}
\end{equation*}
for some structure $X \in \@K$ and embeddings $h : B -> X$ and $k : C -> X$.
\end{itemize}
Common examples of classes $\@K$ with these properties include: finitely generated groups; posets; nontrivial Boolean algebras; finite fields of fixed characteristic $p$.

From the AP (and optionally the JEP), one has various ``generalized amalgamation properties'', whereby more complicated diagrams (of embeddings) can be completed into commutative diagrams (of embeddings), e.g., the following diagram by two uses of AP:
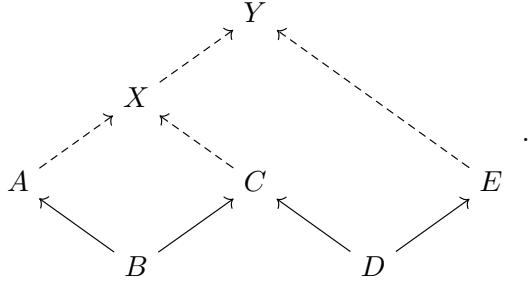
\begin{equation*}
\begin{tikzcd}[row sep=1.5em]
&& Y \\
& X \urar[dashed] \\
A \urar[dashed] && C \ular[dashed] && E \ar[dashed]{uull} \\
& B \ular \urar && D \ular \urar
\end{tikzcd}.
\end{equation*}
However, the following diagram cannot be completed using just the AP (and the JEP):
\begin{equation}
\label{eqn:bowtie}
\begin{tikzcd}[column sep=1em]
A & ~ & B \\
C \uar{f} \ar{urr}[pos=.8]{h} && D \ar[crossing over]{ull}[pos=.8,swap]{g} \uar[swap]{k}
\end{tikzcd}.
\end{equation}
For example, take $\@K =$ the class of finite sets, $A = C = D = 1$, $B = 2$, and $h \ne k$.  This leads to the following
\begin{definitionnamed}{Question}
Can we characterize the shapes of the diagrams which can always be completed using the AP, i.e., the ``generalized amalgamation properties'' which are implied by the AP?  If so, is such a characterization decidable?
\end{definitionnamed}

This question concerns only abstract properties of diagrams and arrows, hence is naturally phrased in the language of category theory.  Let $\*C$ be a category.  Recall that a \defn{cocone} over a diagram in $\*C$ consists of an object $X \in \*C$, together with morphisms $f_A : A -> X$ in $\*C$ for each object $A$ in the diagram, such that the morphisms $f_A$ commute with the morphisms in the diagram; this is formally what it means to ``complete'' a diagram.  Recall also that a \defn{colimit} of a diagram is a universal cocone, i.e., one which admits a unique morphism to any other cocone.  (See \cref{sec:prelims} for the precise definitions.)

We say that $\*C$ has the AP if every pushout diagram (i.e., diagram of shape 
$\bullet <- \bullet -> \bullet$) in $\*C$ has a cocone, and that $\*C$ has the JEP\footnote{We borrow this terminology from model theory, even when not assuming that morphisms are ``embeddings'' in any sense; in particular, we do not assume that every morphism in $\*C$ is monic (although see \cref{sec:general}).} if every diagram in $\*C$ consisting of finitely many objects (without any arrows) has a cocone.  When $\*C$ is the category of structures in a class $\@K$ and embeddings between them, this recovers the model-theoretic notions defined above.  Category-theoretic questions arising in Fraïssé theory have been considered previously in the literature; see e.g., \cite{K}, \cite{C}.

The possibility of answering the above question in the generality of an arbitrary category is suggested by an analogous result of Paré \cite{P} (see \cref{thm:pushout} below), which characterizes the diagram shapes over which a colimit may be built by pushouts (i.e., colimits of pushout diagrams).  There, the crucial condition is that the diagram shape must be \defn{simply-connected} (see \cref{defn:pi1}); failure to be simply-connected is witnessed by the \defn{fundamental groupoid} of the diagram shape, whose morphisms are ``paths up to homotopy''.  For example, the fundamental groupoid of the shape of \eqref{eqn:bowtie} is equivalent to $\#Z$, with generator given by the ``loop'' $A <- D -> B <- C -> A$.

However, simply-connectedness of a diagram's shape does not guarantee that a cocone over it may be built using only the AP (see \cref{ex:boat} below).  Intuitively, the discrepancy with Paré's result is because the universal property of a pushout allows it to be used in more ways to build further cocones.  Simply-connectedness nonetheless plays a role in the following characterization, which is the main result of this paper:

\begin{theorem}
\label{thm:amalgam}
Let $\*I$ be a finitely generated category.  The following are equivalent:
\begin{enumerate}
\item[(i)]  Every $\*I$-shaped diagram in a category with the AP and the JEP has a cocone.
\item[(ii)]  Every $\*I$-shaped diagram in the category of sets and injections has a cocone.  (When $\*I$ is finite, it suffices to consider finite sets.)
\item[(iii)]  $\*I$ is upward-simply-connected (see \cref{defn:L}).
\end{enumerate}
When $\*I$ is a finite poset, these are further equivalent to:
\begin{enumerate}
\item[(iv)]  Every upward-closed subset of $\*I$ is simply-connected.
\item[(v)]  $\*I$ is forest-like (see \cref{defn:tree}; this means $\*I$ is built via some simple inductive rules).
\end{enumerate}
Similarly, every $\*I$-shaped diagram in a category with the AP has a cocone, iff $\*I$ is connected and any/all of (ii), (iii) (also (iv), (v) if $\*I$ is a poset) hold.
\end{theorem}

A corollary of our proof yields a simple decision procedure for these conditions (for finite $\*I$).  This is somewhat surprising, because Paré's result (\cref{thm:pushout}) implies that the analogous question of whether every $\*I$-shaped diagram has a colimit in a category with pushouts is \emph{un}decidable.

This paper is organized as follows.  In \cref{sec:prelims}, we fix notations and review some categorical concepts.  In \cref{sec:setinj}, we introduce an invariant $\@L(\*I)$, similar to the fundamental groupoid, and use it to prove the equivalence of (ii) and (iii) in \cref{thm:amalgam} for arbitrary small (not necessarily finitely generated) $\*I$.  In \cref{sec:posetal}, we analyze upward-simply-connected posets in more detail, deriving the conditions (iv) and (v) equivalent to (iii) and proving that they imply (i) when $\*I$ is a finite poset.  In \cref{sec:general}, we remove this restriction on $\*I$ and complete the proof.  Finally, in \cref{sec:decidability}, we discuss decidability of the equivalent conditions in \cref{thm:amalgam} and of the analogous conditions in Paré's result.

\medskip
\textit{Acknowledgments.}  We would like to thank Alexander Kechris for providing some feedback on an earlier draft of this paper.

\section{Preliminaries}
\label{sec:prelims}

We begin by fixing notations and terminology for some basic categorical notions; see \cite{ML}.

For a category $\*C$ and objects $X, Y \in \*C$, we denote a morphism between them by $f : X --->[\*C]{} Y$.  We use the terms \defn{morphism} and \defn{arrow} interchangeably.

We use $\*{Set}$ to denote the category of sets and functions, $\*{Inj}$ to denote the category of sets and injections, and $\*{PInj}$ to denote the category of sets and partial injections.  We use $\*{Cat}, \*{Gpd}$ to denote the categories of small categories, resp., small groupoids.\footnote{We will generally ignore size issues; it is straightforward to check that except where smallness is explicitly assumed, the following definitions and results work equally well for large categories.}

We regard a preordered set $(\*I, \le)$ as a category where there is a unique arrow $I --->[\*I] J$ iff $I \le J$.

We say that a category $\*C$ is \defn{monic} if every morphism $f : X --->[\*C]{} Y$ in it is monic (i.e., if $f \circ g = f \circ h$ then $g = h$, for all $g, h : Z --->[\*C]{} X$).  Similarly, $\*C$ is \defn{idempotent} if every endomorphism $f : X --->[\*C]{} X$ is idempotent (i.e., $f \circ f = f$).

A category $\*I$ is \defn{finitely generated} if there are finitely many arrows in $\*I$ whose closure under composition is all arrows in $\*I$.  Note that such $\*I$ necessarily has finitely many objects, and that a preorder is finitely generated iff it is finite.

For a category $\*C$ and a small category $\*I$, a \defn{diagram} of shape $\*I$ in $\*C$ is simply a functor $F : \*I -> \*C$.  A \defn{cocone} $(X, \-f)$ over a diagram $F$ consists of an object $X \in \*C$ together with a family of morphisms $\-f = (f_I : F(I) --->[\*C]{} X)_{I \in \*I}$, such that for each $i : I --->[\*I]{} J$, we have $f_I = f_J \circ F(i)$.  A \defn{morphism between cocones} $(X, \-f)$ and $(Y, \-g)$ over the diagram $F$ is a morphism $h : X --->[\*C]{} Y$ such that for each object $I \in \*I$, we have $h \circ f_I = g_I$.  A cocone $(X, \-f)$ over $F$ is a \defn{colimit} of $F$ if it is initial in the category of cocones over $F$, i.e., for any other cocone $(Y, \-g)$ there is a unique cocone morphism $h : (X, \-f) -> (Y, \-g)$; in this case we write $X = \injlim F$, and usually use a letter like $\iota_I$ for the cocone maps $f_I$.
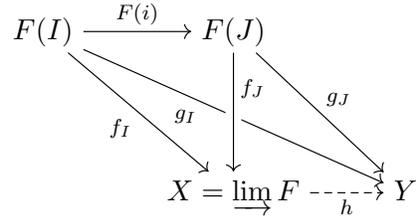
\begin{equation*}
\begin{tikzcd}
F(I) \rar{F(i)} \drar[swap]{f_I} \ar{drr}[swap,pos=.4]{g_I} & F(J) \dar[crossing over][pos=.3]{f_J} \drar{g_J} \\[2em]
& X = \injlim F \rar[dashed][swap]{h} & Y
\end{tikzcd}
\end{equation*}

As mentioned above, a category $\*C$ has the \defn{amalgamation property (AP)} if every \defn{pushout diagram} (i.e., diagram of shape $\bullet <- \bullet -> \bullet$) in $\*C$ has a cocone (colimits of such diagrams are called \defn{pushouts}), and $\*C$ has the \defn{joint embedding property (JEP)} (regardless of whether $\*C$ is monic) if every finite coproduct diagram (i.e., diagram of finite discrete shape) in $\*C$ has a cocone.  (So the empty category does not have the JEP.)

A category $\*C$ is \defn{connected} if it has exactly one \defn{connected component}, where $X, Y \in \*C$ are in the same connected component if they are joined by a zigzag of morphisms
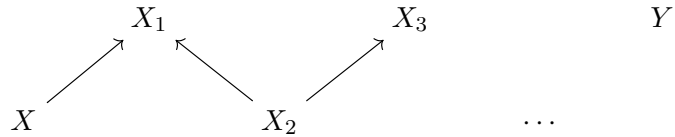
\begin{equation*}
\begin{tikzcd}
& X_1 && X_3 && Y \\
X \urar && X_2 \ular \urar && \dotsb
\end{tikzcd}
\end{equation*}
(So the empty category is not connected.)  We use $\pi_0(\*C)$ to denote the set (or class, if $\*C$ is large) of connected components of $\*C$.  Note that in the presence of the AP, connectedness is equivalent to the JEP, since the AP may be used to turn ``troughs'' into ``peaks'' in a zigzag.

\subsection{Simply-connected categories}

\begin{definition}
\label{defn:pi1}
The \defn{fundamental groupoid} of a category $\*I$, denoted $\pi_1(\*I)$, is the groupoid freely generated by $\*I$ (as a category).  Thus $\pi_1(\*I)$ has the same objects as $\*I$, while its morphisms are words made up of the morphisms in $\*I$ together with their formal inverses, modulo the relations which hold in $\*I$ (and the relations which say that the formal inverses are inverses).

We say that $\*I$ is \defn{simply-connected} if $\pi_1(\*I)$ is an equivalence relation, i.e., has at most one morphism between any two objects.
\end{definition}

\begin{remark}
\label{rmk:nerve}
There is also a topological definition: $\pi_1(\*I)$ is the same as the fundamental groupoid of the (simplicial) nerve of $\*I$; see \cite[\S1]{Q} for the general case.  When $\*I$ is a poset, the \defn{nerve} of $\*I$ can be defined as the (abstract) simplicial complex whose $n$-simplices are the chains of cardinality $n+1$ in $\*I$; see \cite[1.4.4, \S2.4]{B} (in which the nerve is called the \emph{order complex}).
\end{remark}

We now state Paré's result \cite{P}, mentioned in the Introduction, characterizing colimits which can be built using pushouts:

\begin{theorem}[Paré]
\label{thm:pushout}
Let $\*I$ be a finitely generated category.  The following are equivalent:
\begin{enumerate}
\item[(i)]  Every $\*I$-shaped diagram in a category with pushouts has a colimit.
\item[(ii)]  $\*I$ is simply-connected.
\end{enumerate}
\end{theorem}

\begin{example}
\label{ex:boat}
Let $\*I$ be the shape of the diagram \eqref{eqn:bowtie} in the Introduction.  As mentioned there, $\pi_1(\*I)$ is equivalent to $\#Z$ (i.e., it is connected and its automorphism group at each object is $\#Z$).  Now let $\*J \supseteq \*I$ be the shape of the (commuting) diagram
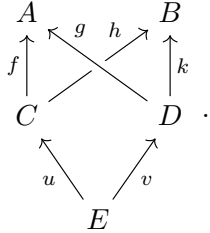
\begin{equation*}
\begin{tikzcd}[column sep=1em]
A && B \\
C \uar{f} \ar{urr}[pos=.8]{h} && D \ar[crossing over]{ull}[pos=.8,swap]{g} \uar[swap]{k} \\
& E \ular{u} \urar[swap]{v}
\end{tikzcd}.
\end{equation*}
Unlike $\*I$, $\*J$ is simply-connected.  Thus by \cref{thm:pushout}, a colimit of a $\*J$-shaped diagram can be constructed out of pushouts.  However, a cocone over a $\*J$-shaped diagram cannot necessarily be constructed from the AP: take $A = C = D = 1$, $B = 2$, and $h \ne k$ in $\*{Inj}$ as in the Introduction, and $E = \emptyset$.  Note that there is no contradiction, since $\*{Inj}$ does \emph{not} have pushouts.

The ``reason'' that $\*J$ is simply-connected even though $\*I$ is not is that the generating ``loop'' $f \circ h^{-1} \circ k \circ g^{-1}$ in $\pi_1(\*I)$ becomes trivial in $\pi_1(\*J)$:
\begin{align*}
f \circ h^{-1} \circ k \circ g^{-1}
&= (f \circ u) \circ (u^{-1} \circ h^{-1}) \circ k \circ g^{-1} \\
&= (g \circ v) \circ (v^{-1} \circ k^{-1}) \circ k \circ g^{-1}
= 1_A.
\end{align*}
This suggests that to characterize when a $\*J$-shaped diagram has a cocone in any category with the AP, we need a finer invariant than $\pi_1$, which does not allow the use of $u, v$ to simplify the loop $f \circ h^{-1} \circ k \circ g^{-1}$ above.
\end{example}

\subsection{Inverse categories}

\begin{definition}
An \defn{inverse category} is a category $\*C$ such that every morphism $f : X --->[\*C]{} Y$ has a unique \defn{pseudoinverse} $f^{-1} : Y --->[\*C]{} X$ obeying $f \circ f^{-1} \circ f = f$ and $f^{-1} \circ f \circ f^{-1} = f^{-1}$. 

We write $\*{InvCat}$ for the category of small inverse categories.
\end{definition}

For basic properties of inverse categories, see e.g., \cite[\S2]{Li}; the one-object case of \emph{inverse monoids} is well-known in semigroup theory \cite{La}.  Here are some elementary facts about inverse categories we will use without mention:

\begin{lemma}
Let $\*C$ be an inverse category.
\begin{itemize}
\item  $f |-> f^{-1}$ is an involutive functor $\*C^\op -> \*C$.
\item  Idempotents in $\*C$ commute.
\item  $f : X --->[\*C]{} Y$ is monic iff it is split monic iff $f^{-1} \circ f = 1_X$.
\end{itemize}
\end{lemma}

The archetypical example of an inverse category is $\*{PInj}$, the category of sets and partial injections (where $f^{-1}$ is given by the partial inverse of $f$).  In fact, the axioms of an inverse category capture precisely the algebraic properties of $\*{PInj}$, in the sense that we have the following representation theorem, generalizing the Wagner-Preston representation theorem for inverse semigroups and the Yoneda lemma (see \cite[2.5]{Li}, \cite[3.8]{CL}):

\begin{theorem}
\label{thm:invcat-yoneda}
Let $\*C$ be a small inverse category.  We have an embedding functor
\begin{align*}
\Psi_\*C : \*C &--> \*{PInj} \\
X &|--> \sum_{Z \in \*C} \Hom_\*C(Z, X) \\
(X --->[\*C]{f} Y) &|--> \left(\begin{aligned}
\sum_Z f^{-1} \circ f \circ \Hom_\*C(Z, X) &\overset{\sim}{-->} \sum_Z f \circ f^{-1} \circ \Hom_\*C(Z, Y) \\
g &|--> f \circ g
\end{aligned}\right).
\end{align*}
Here $\sum$ denotes disjoint union, and $f^{-1} \circ f \circ \Hom_\*C(Z, X)$ denotes the set of all composites $f^{-1} \circ f \circ g$ for $g : Z --->[\*C]{} X$, equivalently the set of all $g : Z --->[\*C]{} X$ such that $g = f^{-1} \circ f \circ g$ (and similarly for $f \circ f^{-1} \circ \Hom_\*C(Z, Y)$).
\end{theorem}

\section{Amalgamating sets}
\label{sec:setinj}

In this section, we characterize the small categories $\*I$ such that every $\*I$-shaped diagram in $\*{Inj}$ has a cocone.  We begin with the following easy observation:

\begin{lemma}
\label{thm:setinj-colim}
A diagram $F : \*I -> \*{Inj}$ has a cocone iff the colimit of $F$ in $\*{Set}$ is such that the canonical maps $\iota_I : F(I) -> \injlim F$ are injective, for all $I \in \*I$.
\end{lemma}
\begin{proof}
If the $\iota_I$ are injective, then $(\injlim F, (\iota_I)_{I \in \*I})$ is a cocone in $\*{Inj}$.  Conversely, if $F$ has a cocone $(X, (f_I)_{I \in \*I})$ in $\*{Inj}$, then the unique cocone morphism $g : \injlim F -> F$ is such that $g \circ \iota_I = f_I$ is injective for each $I$, hence $\iota_I$ is injective for each $I$.
\end{proof}

Now recall that for a diagram $F : \*I -> \*{Inj}$, the standard construction of $\injlim F$ in $\*{Set}$ is as the quotient of the disjoint sum:
\begin{align*}
\injlim F := (\sum_{I \in \*I} F(I))/\{(x, F(i)(x)) \mid i : I --->[\*I]{} J,\, x \in F(I)\}.
\end{align*}
Two elements $x, y \in F(I)$ are thus identified iff they are connected by a zigzag
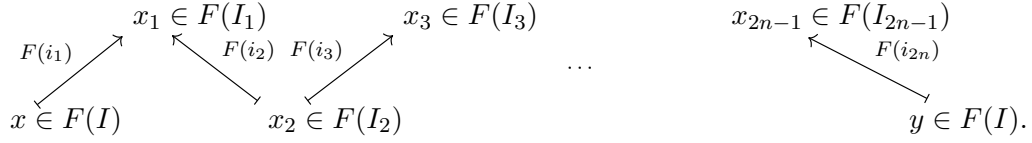
\begin{equation*}
\begin{tikzcd}[column sep=3em]
& x_1 \mathrlap{{}\in F(I_1)} && x_3 \mathrlap{{}\in F(I_3)} & ~ & x_{2n-1} \mathrlap{{}\in F(I_{2n-1})} \\
x \mathrlap{{}\in F(I)} \urar[mapsto]{F(i_1)} && x_2 \mathrlap{{}\in F(I_2)} \ular[mapsto][swap]{F(i_2)} \urar[mapsto]{F(i_3)} && \hspace{3em} \uar[draw=none][pos=.5,anchor=center]{\textstyle\dotsm} && y \in F(I) \ular[mapsto][swap]{F(i_{2n})}.
\end{tikzcd}
\end{equation*}
Since the $F(i_k)$ are injective, the endpoint $y$ of this zigzag is determined by $x$ together with the ``path'' $I --->[\*I]{i_1} I_1 <---[\*I]{i_2} I_2 --->[\*I]{i_2} \dotsb <---[\*I]{i_{2n}} I$ in $I$; in other words, $F$ induces an action of such ``paths'' via partial injections between sets.  This motivates defining the category of such ``paths'', while keeping in mind that they will be acting via partial injections:

\begin{definition}
\label{defn:L}
The \defn{left fundamental inverse category} of a category $\*I$, denoted $\@L(\*I)$, is the inverse category freely generated by $\*I$ such that every morphism in $\*I$ becomes monic in $\@L(\*I)$.

Thus, we have a functor $\eta = \eta_\*I : \*I -> \@L(\*I)$, such that $\@L(\*I)$ is generated by the morphisms $\eta(i), \eta(i)^{-1}$ for $i : I --->[\*I]{} J$, such that $\eta(i)^{-1} \circ \eta(i) = 1_I$ for each such $i$, and such that any other functor $F : \*I -> \*C$ into an inverse category with $F(i)^{-1} \circ F(i) = 1_{F(I)}$ for each $i$ factors uniquely through $\eta$.  (We write $i$ for $\eta(i)$ when there is no risk of confusion.)

We say that $\*I$ is \defn{upward-simply-connected} if $\@L(\*I)$ is idempotent.  (See \cref{thm:usc-etc} below for equivalent conditions when $\*I$ is a poset.)
\end{definition}

Thus $\@L$ extends in an obvious manner to a functor $\*{Cat} -> \*{InvCat}$, which is left adjoint to the functor $\@S : \*{InvCat} -> \*{Cat}$ taking an inverse category to its subcategory of monomorphisms.  To see that $\@L(\*I)$ is a finer invariant than $\pi_1(\*I)$, note that the forgetful functor $\*{Gpd} -> \*{Cat}$ (to which $\pi_1$ is left adjoint) factors through $\@S$; indeed, a groupoid is precisely an inverse category in which every morphism is monic.  Thus $\pi_1(\*I)$ is a quotient of $\@L(\*I)$ (by the least congruence which makes every arrow monic).  Since idempotents in a groupoid are identities, we get

\begin{corollary}
\label{thm:usc-sc}
If a category $\*I$ is upward-simply-connected, then it is simply-connected.  \qed
\end{corollary}

\begin{proposition}
\label{thm:setinj-amalgam}
Let $\*I$ be a small category.  The following are equivalent:
\begin{enumerate}
\item[(i)]  Every diagram $F : \*I -> \*{Inj}$ has a cocone.
\item[(ii)]  The diagram $\Psi_{\@L(\*I)} \circ \eta_\*I : \*I -> \*{Inj}$ has a cocone, where $\Psi_{\@L(\*I)} : \@L(\*I) -> \*{PInj}$ is the embedding from \cref{thm:invcat-yoneda} (whose restriction along $\eta_\*I : \*I -> \@L(\*I)$ lands in the subcategory $\*{Inj}$).
\item[(iii)]  $\*I$ is upward-simply-connected.
\end{enumerate}
\end{proposition}
\begin{proof}
First we remark on why $\Psi_{\@L(\*I)} \circ \eta_\*I$ in (ii) lands in $\*{Inj}$.  This is because every morphism $i$ in $\*I$ becomes monic in $\@L(\*I)$, hence in $\*{PInj}$; and the monomorphisms in $\*{PInj}$ are precisely the total injections.

(i)$\implies$(ii) is obvious.

(ii)$\implies$(iii): Let $f : I --->[\@L(\*I)]{} I$ be an endomorphism.  To show that $f$ is idempotent, it suffices to show that $f \circ f^{-1} \circ f = f^{-1} \circ f$, since then $f \circ f = f \circ (f \circ f^{-1} \circ f) = f \circ (f^{-1} \circ f) = f$.  Since $\@L(\*I)$ is generated by the morphisms in $\*I$ and their pseudoinverses, we have
\begin{align*}
f = i_{2n}^{-1} \circ i_{2n-1} \circ \dotsb \circ i_3 \circ i_2^{-1} \circ i_1
\end{align*}
for some zigzag ``path'' in $\*I$
\begin{equation*}
\begin{tikzcd}[column sep=2em]
& I_1 && I_3 & ~ & I_{2n-1} \\
I = I_0 \urar{i_1} && I_2 \ular[swap]{i_2} \urar{i_3} && \uar[draw=none][pos=.5,anchor=center]{\textstyle\dotsm} && I_{2n} = I. \ular[swap]{i_{2n}}
\end{tikzcd}
\end{equation*}
This yields a zigzag (with some obvious abbreviations for clarity)
\begin{equation*}
\begin{tikzcd}[column sep=1em]
& i_1 f^{-1} f \mathrlap{{}\in \Psi(I_1)} &
& i_3 i_2^{-1} i_1 f^{-1} f \mathrlap{{}\in \Psi(I_3)} & ~
& i_{2n-1} \dotsm i_2^{-1} i_1 f^{-1} f \mathrlap{{}\in \Psi(I_{2n-1})} \\
f^{-1} f \mathrlap{{}\in \Psi(I)} \urar[mapsto]{i_1} &&
i_2^{-1} i_1 f^{-1} f \mathrlap{{}\in \Psi(I_2)} \ular[mapsto][swap]{i_2} \urar[mapsto]{i_3} &&
\hspace{2em} \uar[draw=none][pos=.5,anchor=center]{\textstyle\dotsm} &&
f f^{-1} f \in \Psi(I) \ular[mapsto][swap]{i_{2n}}
\end{tikzcd}
\end{equation*}
where the even-numbered mappings are by the following calculation:
\begin{align*}
i_{2k} (i_{2k}^{-1} i_{2k-1} \dotsm i_2^{-1} i_1 f^{-1} f)
&= (i_{2k} i_{2k}^{-1}) (i_{2k-1} \dotsm i_1 \overbrace{i_1^{-1} \dotsm i_{2k-1}^{-1}) i_{2k} \dotsm i_{2n}}^{f^{-1}} f \\
&= (i_{2k-1} \dotsm i_1 i_1^{-1} \dotsm i_{2k-1}^{-1}) (i_{2k} i_{2k}^{-1}) i_{2k} \dotsm i_{2n} f \\
&= (i_{2k-1} \dotsm i_1 i_1^{-1} \dotsm i_{2k-1}^{-1}) i_{2k} \dotsm i_{2n} f \\
&= i_{2k-1} \dotsm i_1 f^{-1} f.
\end{align*}
By the discussion following \cref{thm:setinj-colim}, this last zigzag implies that $f^{-1} f = f f^{-1} f$ in $\injlim (\Psi \circ \eta)$, whence by (ii) and \cref{thm:setinj-colim}, the same equality holds in $\Psi(I)$ and hence in $\@L(\*I)$, as desired.

(iii)$\implies$(i): Let $F : \*I -> \*{Inj}$ be a diagram.  By the universal property of $\@L(\*I)$, $F$ extends along $\eta$ to a functor $\~F : \@L(\*I) -> \*{PInj}$ such that $\~F \circ \eta = F : \*I -> \*{PInj}$.  This functor $\~F$ takes a pseudoinverse $i^{-1} : J --->[\@L(\*I)]{} I$, for $i : I --->[\*I]{} J$, to the (partial) inverse of $F(i)$, hence takes a ``path'' $f = i_{2n}^{-1} \circ i_{2n-1} \circ \dotsb \circ i_1 : I --->[\@L(\*I)]{} I$ to the partial injection $\~F(f) : F(I) --->[\*{PInj}]{} F(I)$ mapping $x \in F(I)$ to the endpoint $y$ of the zigzag in the remarks following \cref{thm:setinj-colim}.  Since every $f : I --->[\@L(\*I)]{} I$ is idempotent, so is every $\~F(f)$.  Since the idempotents in $\*{PInj}$ are precisely the partial identity functions, it follows (by the remarks following \cref{thm:setinj-colim}) that the canonical maps $\iota_I : F(I) -> \injlim F$ are injective.
\end{proof}

This proves (ii)$\iff$(iii) in \cref{thm:amalgam}, for arbitrary small $\*I$; the parenthetical in (ii) follows from

\begin{lemma}
Let $\*I$ be a finite category.  If a diagram $F : \*I -> \*{Inj}$ does not have a cocone, then some diagram $F'$ which is a pointwise restriction of $F$ to finite subsets also does not have a cocone.
\end{lemma}
\begin{proof}
If $F$ does not have a cocone, then there is some $I \in \*I$ and $x \ne y \in F(I)$ which are identified in $\injlim F$.  Take a zigzag $x = x_0 |-> x_1 <-| x_2 |-> \dotsb <-| x_{2n} = y$ witnessing that $x, y$ are identified in $\injlim F$, as in the remarks following \cref{thm:setinj-colim}, where $x_k \in F(I_k)$.  Put $F'(J) := \{F(i)(x_k) \mid i : I_k --->[\*I]{} J\}$.
\end{proof}

\section{Posetal diagrams}
\label{sec:posetal}

In this section, we prove that (iii)$\implies$(i) in \cref{thm:amalgam} for finite posets $\*I$.  To do so, we first examine the structure of upward-simply-connected posets; this will lead to the conditions (iv) and (v) in \cref{thm:amalgam}.

\begin{lemma}
\label{thm:L-cosieve-faithful}
Let $\*I$ be any category and $\*J \subseteq \*I$ be an upward-closed subcategory (or \defn{cosieve}), i.e., a full subcategory such that if $I \in \*J$ and $i : I --->[\*I]{} J$ then $i, J \in \*J$.  Then the canonical induced functor $\@L(\*J) -> \@L(\*I)$ is faithful.
\end{lemma}
\begin{proof}
Let $\*K$ be the inverse category obtained by taking $\@L(\*J)$, adding the objects in $\*I \setminus \*J$, and freely adjoining zero morphisms between every pair of objects (taking the zero morphism from an object in $\*I \setminus \*J$ to itself as the identity).  Let $F : \*I -> \*K$ send morphisms $i$ in $\*J$ to $\eta_\*J(i)$ and all other morphisms to $0$.  Then $F$ sends all morphisms to monomorphisms, hence extends along $\eta_\*I : \*I -> \@L(\*I)$ to $\~F : \@L(\*I) -> \*K$.
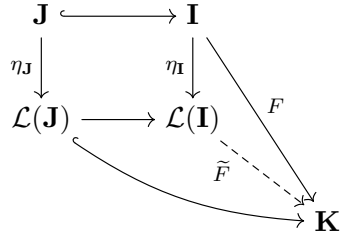
\begin{equation*}
\begin{tikzcd}
\*J \dar[swap]{\eta_\*J} \rar[hook] & \*I \dar[swap]{\eta_\*I} \ar{ddr}{F} \\
\@L(\*J) \ar[hook,bend right=15]{drr} \rar & \@L(\*I) \drar[dashed][swap,pos=.2]{\~F} \\
&& \*K
\end{tikzcd}
\end{equation*}
The composite $\@L(\*J) -> \@L(\*I) --->{\~F} \*K$ is equal to the inclusion $\@L(\*J) \subseteq \*K$, because it takes morphisms $\eta_\*J(i)$ for $i$ in $\*J$ to $F(i) = \eta_\*J(i)$.  It follows that $\@L(\*J) -> \@L(\*I)$ must be faithful, as desired.
\end{proof}

\begin{corollary}
\label{thm:usc-cosieve-sc}
If $\*I$ is upward-simply-connected, then every upward-closed subcategory $\*J \subseteq \*I$ is simply-connected.
\end{corollary}
\begin{proof}
By \cref{thm:L-cosieve-faithful} and \cref{thm:usc-sc}.
\end{proof}

We will show that for finite posets $\*I$, the two conditions in \cref{thm:usc-cosieve-sc} are equivalent to each other and to the following combinatorial notion:

\begin{definition}
\label{defn:tree}
The class of finite \defn{tree-like} posets is defined inductively by the following rule:
\begin{enumerate}
\item[($*$)]  if $\*K_1, \dotsc, \*K_n$ are finite tree-like posets, and $\*U_k \subseteq \*K_k$ is a connected upward-closed subset for each $k$, then a new finite tree-like poset is formed by taking the disjoint union of the $\*K_k$'s and adjoining a single point which lies below each $\*U_k$.
\end{enumerate}
A finite \defn{forest-like} poset is a finite disjoint union of tree-like posets.
\end{definition}

\begin{proposition}
\label{thm:posetal-cosieve-sc-decomp}
Let $\*I$ be a finite connected simply-connected poset and $I \in \*I$ be minimal.  Then for each connected component $\*K \in \pi_0(\*I \setminus \{I\})$, the subposet $\*K \cap \up I$ (where $\up I := \{J \in \*I \mid J \ge I\}$) is connected.
\end{proposition}
\begin{proof}
This is straightforward to show using topological arguments, by considering the nerve of $\*I$ as in \cref{rmk:nerve}; however, in the interest of keeping this paper self-contained, we will give a more elementary proof.

Let $\pi_0(\*K \cap \up I) = \{\*U_1, \dotsc, \*U_m\}$ and $\pi_0(\*K \setminus \up I) = \{\*V_1, \dotsc, \*V_n\}$.  Note that we must have $m \ge 1$, since $\*K \ne \emptyset$ and $\*I$ is connected.  Consider the partition of connected subposets
\begin{align*}
\@P := \{\*I \setminus \*K\} \cup \pi_0(\*K \cap \up I) \cup \pi_0(\*K \setminus \up I)
\end{align*}
of $\*I$, ordered by $\*A \le \*B \iff \exists A \in \*A,\, B \in \*B.\, A \le B$ for $\*A, \*B \in \@P$.  Then $\@P$ is connected (because $\*I$ is), $\pi_0(\*K \cap \up I) \cup \pi_0(\*K \setminus \up I) \subseteq \@P$ is connected (because $\*K$ is), and we have two antichains $\{\*I \setminus \*K\} \cup \pi_0(\*K \setminus \up I)$ and $\pi_0(\*K \cap \up I)$, with only elements of the former below elements of the latter (and $\*I \setminus \*K$ below every element of the latter).  So the Hasse diagram of $\@P$ looks like
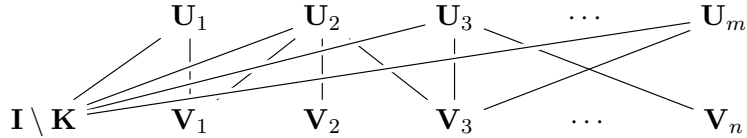
\begin{equation*}
\begin{tikzcd}[every arrow/.append style={-,crossing over clearance=.7ex}]
& \*U_1 \dar & \*U_2 \dlar \dar \drar & \*U_3 \dar \ar{drr} & \dotsb & \*U_m \ar{dll} \\
\*I \setminus \*K \urar \ar[crossing over]{urr} \ar[crossing over]{urrr} \ar[crossing over]{urrrrr} &
\*V_1 & \*V_2 & \*V_3 & \dotsb & \*V_n
\end{tikzcd}
\end{equation*}
Using that each element of $\@P$ is connected, it is easy to check that the functor $\pi_1(\*I) -> \pi_1(\@P)$ induced by the quotient map $\*I -> \@P$ is full.  So since $\*I$ is simply-connected, so must be $\@P$.  But since $\@P$ has no chains of cardinality $> 2$, $\pi_1(\@P)$ is just the graph-theoretic fundamental groupoid of its Hasse diagram (depicted above).  So for $\@P$ to be simply-connected, its Hasse diagram must be acyclic, which clearly implies $|\pi_0(\*K \cap \up I)| = m = 1$, as desired.
\end{proof}

\begin{corollary}
\label{thm:cosieve-sc-tree}
Let $\*I$ be a finite (connected) poset such that every upward-closed subset is simply-connected.  Then $\*I$ is forest-like (tree-like).  \qed
\end{corollary}

\begin{proposition}
\label{thm:amalgam-tree}
Let $\*I$ be a finite tree-like poset.  Then for every category $\*C$ with the AP, every diagram $F : \*I -> \*C$ has a cocone.
\end{proposition}
\begin{proof}
By induction on the construction of $\*I$, there is $I \in \*I$ minimal such that $\*I \setminus \{I\}$ is the disjoint union of tree-like posets $\*K_1, \dotsc, \*K_n$ and $\*U_k := \*K_k \cap \up I$ is connected for each $k$.  For each $k$, there is some $U_k \in \*U_k$ since $\*U_k$ is connected, and there is a cocone $(X_k, (f^k_K)_{K \in \*K_k})$ over $F|\*K_k$ by the induction hypothesis.  By the AP in $\*C$, there is a cocone $(Y, (g_I, g_k)_k)$ over the pushout diagram consisting of the composite maps
\begin{align*}
F(I) --->[\*C]{F(I, U_k)} F(U_k) --->[\*C]{f^k_{U_k}} X_k
\end{align*}
(where $F(I, U_k)$ denotes $F$ applied to the unique morphism $I --->[\*I]{} U_k$) for all $k$, where $g_I : F(I) --->[\*C]{} Y$ and $g_k : X_k --->[\*C]{} Y$ with
\begin{align*}
g_k \circ f^k_{U_k} \circ F(I, U_k) &= g_I.
\end{align*}
For $K \in \*K_k$, let $g_K : F(K) -> Y$ be the composite
\begin{align*}
F(K) --->{f^k_K} X_k --->{g_k} Y.
\end{align*}
We claim that $(Y, (g_J)_{J \in \*I})$ is a cocone over $F$.  For $K < K'$ where $K \in \*K_k$ for some $k$, also $K' \in \*K_k$, whence $g_K = g_k \circ f^k_K = g_k \circ f^k_{K'} \circ F(K, K') = g_{K'} \circ F(K, K')$ because $(X_k, (f^k_K)_{K \in \*K_k})$ is a cocone over $F|\*K_k$.  So we only need to check that for $I < K \in \*K_k$, i.e., $K \in \*U_k$, we have $g_I = g_K \circ F(I, K)$.  By connectedness of $\*U_k$, there is a path $K = K_0 \le K_1 \ge K_2 \le \dotsb \ge K_{2n} = U_k$ in $\*U_k$, whence
\begin{align*}
g_K \circ F(I, K)
&= g_k \circ f^k_{K_0} \circ F(I, K_0) \\
&= g_k \circ f^k_{K_1} \circ F(K_0, K_1) \circ F(I, K_0)
= g_k \circ f^k_{K_1} \circ F(I, K_1) \\
&= g_k \circ f^k_{K_1} \circ F(K_2, K_1) \circ F(I, K_2)
= g_k \circ f^k_{K_2} \circ F(I, K_2) \\
&= \dotsb \\
&= g_k \circ f^k_{K_{2n}} \circ F(I, K_{2n})
= g_k \circ f^k_{U_k} \circ F(I, U_k) = g_I.
\qedhere
\end{align*}
\end{proof}

\begin{corollary}
\label{thm:amalgam-forest}
Let $\*I$ be a finite forest-like poset.  Then for every category $\*C$ with the AP and the JEP, every diagram $F : \*I -> \*C$ has a cocone.
\end{corollary}
\begin{proof}
Find cocones over $F$ restricted to each connected component of $\*I$, then apply JEP.
\end{proof}

\begin{corollary}
\label{thm:usc-etc}
A finite poset $\*I$ is upward-simply-connected, iff every upward-closed subset $\*J \subseteq \*I$ is simply-connected, iff $\*I$ is forest-like.
\end{corollary}
\begin{proof}
By \cref{thm:usc-cosieve-sc}, \cref{thm:cosieve-sc-tree}, \cref{thm:amalgam-forest}, and \cref{thm:setinj-amalgam}.
\end{proof}

This proves (iii)$\iff$(iv)$\iff$(v)$\implies$(i) in \cref{thm:amalgam} for finite posets $\*I$.

\section{General diagrams}
\label{sec:general}

We begin this section by explaining why in a context such as Fraïssé theory, where we are looking at diagrams in a monic category $\*C$ (of embeddings in the case of Fraïssé theory), the only diagram shapes $\*I$ worth considering are posets.

\begin{definition}
The \defn{monic reflection} of a category $\*I$ is the monic category $\@M(\*I)$ freely generated by $\*I$.  Explicitly, $\@M(\*I) = \*I/{\sim}$ for the least congruence $\sim$ on $\*I$ such that $\*I/{\sim}$ is monic.
\end{definition}

Thus for a diagram $F : \*I -> \*C$ in a monic category $\*C$, $F$ factors through $\@M(\*I)$, say as $F' : \@M(\*I) -> \*C$; and pullback of cocones over $F'$ along the projection $\pi : \*I -> \@M(\*I)$ is an isomorphism of categories (between the category of cocones over $F'$ and the category of cocones over $F$).  So in a monic category $\*C$, we may as well only consider diagrams whose shape $\*I$ is monic.  And by the following, an upward-simply-connected monic $\*I$ is necessarily a preorder; clearly $\*I$ may be replaced with an equivalent poset in that case.

\begin{lemma}
\label{thm:incat-mono}
Let $\*C$ be an idempotent inverse category.  Then there is at most one monomorphism between any two objects in $\*C$.
\end{lemma}
\begin{proof}
Let $f, g : X --->[\*C]{} Y$ be two monomorphisms.  Then $f \circ g^{-1} : Y -> Y$ is idempotent, whence $g^{-1} \circ f = f^{-1} \circ f \circ g^{-1} \circ f \circ g^{-1} \circ g = f^{-1} \circ f \circ g^{-1} \circ g = 1_X$.  This implies that $g^{-1}$ is a pseudoinverse of $f$, whence $g^{-1} = f^{-1}$, whence $g = f$.
\end{proof}

\begin{corollary}
\label{thm:usc-mon-posetal}
An upward-simply-connected small monic category $\*I$ is necessarily a preorder.
\end{corollary}
\begin{proof}
Since $\*I$ is monic, we have an embedding
\begin{align*}
\*I &--> \*{Inj} \subseteq \*{PInj} \\
I &|--> \sum_{K \in \*I} (K --->[\*I]{} I) \\
(I --->[\*I]{i} J) &|--> (j |-> i \circ j);
\end{align*}
this embedding factors through the canonical functor $\eta : \*I -> \@L(\*I)$, so $\eta$ must be faithful.  By \cref{thm:incat-mono}, for $i, j : I --->[\*I]{} J$, we have $\eta(i) = \eta(j)$, whence $i = j$.
\end{proof}

Now we check that $\@L(\*I)$ is invariant when passing from $\*I$ to $\@M(\*I)$:

\begin{lemma}
\label{thm:L-M}
For any (small) category $\*I$, $\@L(\*I) \cong \@L(\@M(\*I))$.
\end{lemma}
\begin{proof}
$\@L : \*{Cat} -> \*{InvCat}$ is left adjoint to the composite
\begin{align*}
\*{InvCat} --> \*{MonCat} --> \*{Cat}
\end{align*}
(where $\*{MonCat}$ is the category of (small) monic categories) where the first functor takes an inverse category to its subcategory of monomorphisms, and the second functor is the full inclusion, which has left adjoint $\@M : \*{Cat} -> \*{MonCat}$.
\end{proof}

\begin{corollary}
\label{thm:amalgam-monic}
Let $\*I$ be a connected upward-simply-connected category with finitely many objects.  Then every diagram $F : \*I -> \*C$ in a monic category $\*C$ with the AP has a cocone.
\end{corollary}
\begin{proof}
Since $\*C$ is monic, $F$ factors through $\@M(\*I)$, as $F' : \@M(\*I) -> \*C$, say.  Since $\*I$ is upward-simply-connected, so is $\@M(\*I)$, whence $\@M(\*I)$ is a finite preorder.  Clearly replacing $\@M(\*I)$ with an equivalent poset does not change whether $F'$ has a cocone, so by \cref{thm:amalgam-tree}, $F'$ has a cocone, which induces a cocone over $F$.
\end{proof}

We now deduce the general case of (iii)$\implies$(i) in \cref{thm:amalgam}:

\begin{lemma}
\label{thm:ap-moniq}
Let $\*C$ be a category with the AP.  Then
\begin{align*}
f \sim g \iff \exists h \in \*C.\, h \circ f = h \circ g
\end{align*}
defines a congruence on $\*C$, thus $\@M(\*C) = \*C/{\sim}$.
\end{lemma}
\begin{proof}
Clearly $\sim$ is reflexive, symmetric, and right-compatible.  To check left-compatibility, suppose $f \sim g : X --->[\*C]{} Y$ and $h : Y --->[\*C]{} Z$; we show $h \circ f \sim h \circ g$.  Let $k : Y --->[\*C]{} W$ witness $f \sim g$, so that $k \circ f = k \circ g$.  Then $h \circ f \sim h \circ g$ is witnessed by $l$ such that the following diagram commutes:
\begin{equation*}
\begin{tikzcd}
& V \\
Z \urar[dashed]{l} && W \ular[dashed] \\
& Y \ular{h} \urar[swap]{k}
\end{tikzcd}
\end{equation*}
To check transitivity, suppose $f \sim g \sim h : X --->[\*C]{} Y$; we show $f \sim h$.  Let $k : Y --->[\*C]{} Z$ witness $f \sim g$ and $l : Y --->[\*C]{} W$ witness $g \sim h$.  Then $f \sim h$ is witnessed by an amalgam of $k, l$.
\end{proof}

\begin{proposition}
Let $\*I$ be a finitely generated connected upward-simply-connected category.  Then every diagram $F : \*I -> \*C$ in a category $\*C$ with the AP has a cocone.
\end{proposition}
\begin{proof}
Let $\pi : \*C -> \@M(\*C)$ be the monic reflection of $\*C$.  By \cref{thm:amalgam-monic}, $\pi \circ F : \*I -> \@M(\*C)$ has a cocone $(Y, \-g)$, where $g_I : F(I) --->[\@M(\*C)]{} Y$ for each $I \in \*I$.  Pick for each $I \in \*I$ a lift $f_I : F(I) --->[\*C]{} Y$ of $g_I$.  For each $i : I --->[\*I]{} J$, since $(Y, \-g)$ is a cocone over $\pi \circ F$, we have $\pi(f_I) = g_I = g_J \circ \pi(F(i)) = \pi(f_J \circ F(i))$; by \cref{thm:ap-moniq}, this means there is some $h_i : Y --->[\*C]{} Z_i$ such that $h_i \circ f_I = h_i \circ f_J \circ F(i)$.  Now letting $h : Y --->[\*C]{} Z$ be an amalgam of $h_i$ for all arrows $i$ in some finite generating set of arrows in $\*I$, we have $h \circ f_I = h \circ f_J \circ F(i)$ for all $i : I --->[\*I]{} J$ in the generating set, so $(Z, h \circ f_I)_{I \in \*I}$ is a cocone over $F$.
\end{proof}

\begin{corollary}
Let $\*I$ be a finitely generated upward-simply-connected category.  Then every diagram $F : \*I -> \*C$ in a category $\*C$ with the AP and the JEP has a cocone.  \qed
\end{corollary}

This proves (iii)$\implies$(i) in \cref{thm:amalgam}.  Since (i)$\implies$(ii) is obvious, to complete the proof of the theorem it only remains to check

\begin{lemma}
Let $\*I$ be a category such that every $\*I$-shaped diagram in a category with the AP has a cocone.  Then $\*I$ is connected.
\end{lemma}
\begin{proof}
Consider the diagram $\*I -> \pi_0(\*I)$ where $\pi_0(\*I)$ is regarded as a discrete category.
\end{proof}

\section{Decidability}
\label{sec:decidability}

Suppose we are given a finite category $\*I$ in some explicit form (say, a list of its morphisms and a composition table).  Then our proof of \cref{thm:amalgam} yields a simple procedure for testing whether a ``generalized amalgamation property'' holds for $\*I$-shaped diagrams:

\begin{corollary}
For a finite category $\*I$, it is decidable whether $\*I$ is upward-simply-connected, hence whether every $\*I$-shaped diagram in a category with the AP (and possibly the JEP) has a cocone.
\end{corollary}
\begin{proof}
First, compute the monic reflection $\@M(\*I)$; this can be done in finite time, since $\@M(\*I) = \*I/{\sim}$ for the least congruence $\sim$ such that $\*I/{\sim}$ is monic, and $\sim$ can be computed by taking the equality congruence and closing it under finitely many conditions, which takes finitely many steps since $\*I$ is finite.  If $\@M(\*I)$ is not a preorder, then $\*I$ is not upward-simply-connected by \cref{thm:usc-mon-posetal}.  If $\@M(\*I)$ is a preorder, then replace it with an equivalent poset $\*J$ and recursively test whether $\*J$ is forest-like using \cref{thm:posetal-cosieve-sc-decomp}, i.e., for each connected component $\*K \in \pi_0(\*J)$, test whether $\*K$ is tree-like by picking some minimal $K \in \*K$ and then for each $\*L \in \pi_0(\*K \setminus \{K\})$ testing whether $\*L \cap \up K$ is connected and whether $\*L$ is tree-like.
\end{proof}

This cannot be extended to finitely presented $\*I$, since if $\*I$ is a group regarded as a one-object category then $\@L(\*I) = \*I$ is idempotent iff $\*I$ is trivial, and it is undecidable whether a finite group presentation presents the trivial group (see e.g., \cite[3.4]{Mi}).

We end by pointing out the following simple, but somewhat surprising, consequence of Paré's result, \cref{thm:pushout}, which shows that the analogy between the AP and pushouts breaks down when it comes to decidability:

\begin{corollary}[of \cref{thm:pushout}]
For a finite poset $\*I$, it is undecidable whether $\*I$ is simply-connected, hence whether every $\*I$-shaped diagram in a category with pushouts has a colimit.
\end{corollary}
\begin{proof}[Proof (sketch)]
There is a standard procedure to turn a finite presentation of a group $G$ into a finite $2$-dimensional (abstract) simplicial complex $K$ with fundamental group $G$; then the nerve (see \cref{rmk:nerve}) of the face poset of $K$ is the barycentric subdivision of $K$, hence the face poset of $K$ has fundamental group(oid equivalent to) $G$ (see \cite{B} for details).
\end{proof}

\bigskip
\noindent Department of Mathematics

\noindent California Institute of Technology

\noindent Pasadena, CA 91125

\medskip
\noindent\nolinkurl{rchen2@caltech.edu}

\end{document}